\documentclass[12pt]{amsart}
\usepackage{todonotes}
\usepackage{graphics}
\usepackage{graphicx}
\usepackage{todonotes}

\usepackage{amssymb}
\usepackage{amsmath}
\usepackage{amsthm}
\usepackage{multicol}
\usepackage[english]{babel}
\usepackage{cmap}
\usepackage{url}
\usepackage{setspace}

\usepackage{setspace}
\makeatletter
 
  \@addtoreset{equation}{section}
\makeatother
\makeatletter
\def\tbcaption{\def\@captype{table}\caption}
\makeatother
\newtheorem{theorem}{Theorem}[section]
\newtheorem{lemma}[theorem]{Lemma}
\newtheorem{cor}[theorem]{Corollary}
\theoremstyle{definition}

\newtheorem{proposition}[theorem]{Proposition}

\newtheorem*{lemma*}{Lemma}

\theoremstyle{remark}

\usepackage{amsmath}
\usepackage{amsfonts}
\usepackage{amssymb}
\usepackage{amsthm}
\usepackage{enumerate}

\theoremstyle{plain}

\newtheorem{conj}[theorem]{Conjecture}
\newtheorem*{rem}{Remark}

\title{On Extensions of $D(4)$-triples by Adjoining Smaller Elements}

\author{Marija {Bliznac Trebje\v{s}anin}}
\address{University of Split, Faculty of Science\\ Rudera Boskovica 33, 21000 Split, Croatia}
\email{marbli@pmfst.hr}

\author{Pavao Radi\'{c}}
\address{University of Split, Faculty of Science\\ Rudera Boskovica 33, 21000 Split, Croatia}
\email{pradic@pmfst.hr}

\begin{document}
\begin{abstract} 
This paper examines the problem of obtaining a $D(4)$-quadruple by adding a smaller element to a $D(4)$-triple. We prove some relations between elements of observed hypothetical $D(4)$-quadruples under which a conjecture of the uniqueness of such an extension holds. Also, it is shown that for any $D(4)$-triple there are at most two extensions with a smaller element.
\end{abstract}

\maketitle

\noindent{\it 2020 {Mathematics Subject Classification:}} 11D09, 11B37, 11J68

\noindent{\it Keywords}: Diophantine $m$-tuples, Pellian equations, hypergeometric method.

\section{Introduction}

Let $n\neq0$ be an integer. We call a set of $m$ distinct positive integers a $D(n)$-$m$-tuple, or $m$-tuple with the property $D(n)$, if the product of any two of its distinct elements increased by $n$ is a perfect square.

One of the most frequently studied questions is how large these sets can be. In the classical case $n=1$, first studied by Diophantus, Dujella proved in \cite{duje_kon} that a $D(1)$-sextuple does not exist and that there are at most finitely many quintuples. Over the years, many authors have improved the upper bound for the number of $D(1)$-quintuples. Finally, in \cite{petorke}, He, Togb\'{e} and Ziegler proved the nonexistence of $D(1)$-quintuples. 

Similar conjectures and observations can be made for the case $n=4$ as for the case $n=1$.
In view of this observation, Filipin and the first author proved in \cite{nas2} that a $D(4)$-quintuple also does not exist.

In both cases, $n=1$ and $n=4$, conjectures about the uniqueness of an extension of a triple to a quadruple with a larger element are still open. 

 Consider a $D(4)$-triple $\{a,b,c\}$ and define numbers $d_+$ and $d_-$ with
 $$d_{\pm}=d_{\pm}(a,b,c)=a+b+c+\frac{1}{2}(abc\pm \sqrt{(ab+4)(ac+4)(bc+4)}).$$
 Then $\{a,b,c\}$ can be extended to a $D(4)$-quadruple with an element $d=d_+$, which is a number greater than $\max\{a,b,c\}$.
A $D(4)$-quadruple $\{a,b,c,d_+(a,b,c)\}$ is called a regular quadruple. If $\{a,b,c,d\}$ is a $D(4)$-quadruple such that $d>\max\{a,b,c\}$ and $d\neq d_+(a,b,c)$ then it is called an irregular quadruple.

\begin{conj}\label{conj_1}
Any $D(4)$-triple $\{a,b,c\}$ has a unique extension to a Diophantine quadruple $\{a,b,c,d\}$ by an element $d>\max\{a,b,c\}$.
\end{conj}
 
Results which support this conjecture can be found, for example, in \cite{bf}, \cite{mbt}, \cite{dujram}, \cite{fil_par} and \cite{fht}.

If $a<b<c$ and $c\neq a+b+2r$ then $0<d_-(a,b,c)<c$. Moreover, $\{a,b,c,d_-(a,b,c)\}$ is another extension to a $D(4)$-quadruple. Notice that in this case $c=d_+(a,b,d_-(a,b,c))$.

In the case of $D(1)$-triples, Cipu, Dujella and Fujita studied in \cite{cdf} extensions of the triples by adjoining an element smaller than all elements of the given triple. Similarly, the first author in \cite{mbtap1} initiated a study of the analogous problem in the case of $D(4)$-triples. 

\begin{conj}\label{conj_2}
    Suppose that $\{a_1, b, c, d\}$ is a $D(4)$-quadruple with $a_1<b<c<d$. Then, $\{a_2, b, c, d\}$ is not a $D(4)$-quadruple for any
integer $a_2$ with $a_1\neq a_2<b$.
\end{conj}

Its validity is checked in \cite{mbtap1} for $c < 0.25b^3$ and is also true for $c\geq 39247b^4$ as a consequence of \cite[Theorem 1.6]{mbt}.

The main result of this paper is to give some conditions under which possible counterexamples to Conjecture \ref{conj_2} may exist.

\begin{theorem}\label{glavni}
    Assume that $\{a_1,b,c,d\}$ and $\{a_2,b,c,d\}$ are $D(4)$-quadruples with $a_1<a_2<b<c<d$. Then the following statements hold:
    \begin{enumerate}[1)]
    \item $a_2>4a_1$.
        \item $a_2>0.0251a_1^2$. 
        \item If $a_1\geq 2$ then $b<a_2^2$ and $a_2\geq 317$.
        \item $0.25a_1^2b^3 < c < 0.25a_2^2b^3$.
    \end{enumerate}
\end{theorem}
Note that one of the $D(4)$-quadruples, either $\{a_1, b, c, d\}$ or $\{a_2, b, c, d\}$, must be irregular. If we assume both are regular quadruples, then we would have $d = d_+(a_1, b, c) = d_+(a_2, b, c)$, which implies $a_1 = a_2$, a contradiction. The following
consequence of Theorem \ref{glavni}, 
shows that an even stronger claim must hold, i.e. both of those quadruples must be irregular.

\begin{cor}\label{cor: nereg}
   If $\{a_1,b,c,d\}$ and $\{a_2,b,c,d\}$ are $D(4)$-quadruples, $a_1<a_2<b<c<d$, then both of them are irregular $D(4)$-quadruples.
\end{cor}

It will also be proven that there are at most two extensions of a $D(4)$-triple $\{b, c, d\}$ with a smaller integer.

\begin{cor}\label{cor: najvisedva}
    Let $\{b, c, d\}$ be a $D(4)$-triple. Then, there exist at most two positive integers $a$ with $a < \min\{b,c,d\}$ such that $\{a,b,c,d\}$ is a $D(4)$-quadruple.
\end{cor}

By combining Theorem \ref{glavni} with results from \cite{mbt}, \cite{duje_kon} and \cite{fil_xy4}, we will establish that there exist only finitely many $D(4)$-triples $\{b, c, d\}$ that can be extended to two distinct $D(4)$-quadruples with smaller elements $a_1 < a_2 < b$. Section 7 will be dedicated to proving this statement.

\begin{cor}\label{cor: konacno}
    There are only finitely many quintuples $\{a_1, a_2, b, c, d\}$ with
$a_1 < a_2 < b < c < d$ such that $\{a_1, b, c, d\}$ and $\{a_2, b, c, d\}$ are 
$D(4)$-quadruples.
\end{cor}

From Corollary \ref{cor: nereg} we see that any counterexample to Conjecture \ref{conj_2} gives rise to two irregular $D(4)$-quadruples, which would be in contradiction to Conjecture \ref{conj_1}, which proves the following.
\begin{cor}\label{cor: conject}
Conjecture \ref{conj_1} implies Conjecture \ref{conj_2}.
\end{cor}

To demonstrate these results, we use mainly the approach and concepts from \cite{cdf2} to derive similar outcomes for extensions of a $D(4)$-triple. In Section \ref{sekcija_druga} we define the system of Pellian equations and introduce our main tools used in search of its solution.  Lemma \ref{lemma: gornja_n_1} gives us an upper bound for a solution of the system in terms of elements of the $D(4)$-triple.  
In Section \ref{sekcija_treca}, we clarify the circumstances under which the conditions of Lemma \ref{lemma: gornja_n_1} are met, reinforcing its applicability in the computer implementation of the algorithm crucial to finishing the majority of our proofs. 
More precisely, by considering the results from Lemma \ref{lemma: gornja_n_1}, along with Lemmas \ref{gornje_ograde3} and \ref{lema: donja na n u terminima b i c}, we are left with only a finite number of $D(4)$-triples to study. The remaining portion of the paper is dedicated to the proof of each statement of Theorem \ref{glavni} and its consequences. The proofs for Corollaries \ref{cor: nereg} and \ref{cor: najvisedva} are covered in Section 4. However, the proof for Corollary \ref{cor: konacno} is presented separately in the final section. This separation is necessary as it relies on an auxiliary result that was not previously established for $D(4)$-quadruples.

These results improve \cite{mbtap1} significantly by giving new relations between elements of $D(4)$-quadruples $\{a_1,b,c,d\}$ and $\{a_2,b,c,d\}$.

\section{Pellian equations and bounds for indices of solutions}\label{sekcija_druga}
Assume that $ \{a_1,b,c,d\} $ and $ \{a_2,b,c,d\} $ are $D(4)$-quadruples with $a_1<a_2<b<c<d$. Let $r_1,r_2,s_1,s_2,t,x_1,x_2,y,z$ be positive integers satisfying 
$$
\begin{array}{lll}
    a_1b+4=r_1^2,  &\quad a_1c+4=s_1^2, & \quad a_1d+4=x_1^2,\\
    a_2b+4=r_2^2, &\quad a_2c+4=s_2^2, & \quad a_2d+4=x_2^2, \\
    bc+4=t^2, &\quad bd+4=y^2, & \quad cd+4=z^2.
\end{array}
$$
Considering $x_1,x_2,y,z$ as unknowns, we obtain the following system of Pellian equations:
\begin{align}
    a_1z^2-cx_1^2&=4(a_1-c) \label{prva} \\a_2z^2-cx_2^2&=4(a_2-c) \label{druga} \\
    bz^2-cy^2&=4(b-c) \label{treca}
\end{align}
As described in \cite{fil_xy4}, any positive integer solution to (\ref{druga}) and (\ref{treca}) is given by $z=v_m=w_n$ for some non-negative integers $m,n$, where $\{v_m\}$ and $\{w_n\}$ are recurrent sequences 
\begin{align}
    v_0&=z_{(0)}, \ v_1=\frac{1}{2}(s_2z_{(0)} + cx_{(0)}), \label{niz vm} \ v_{m+2}=s_2v_{m+1}-v_m, \\
    w_0&=z_{(1)}, \ w_1=\frac{1}{2}(tz_{(1)} + cy_{(1)}), \ w_{n+2}=tw_{n+1}-w_n.\label{niz wm}
\end{align}
Moreover, Theorem 1.3 and Lemma 2.14 in \cite{mbt} imply that a solution $z=v_m=w_n$ also satisfies either
\begin{equation}
     m \equiv n \equiv 0 \ (\bmod\ 2) \ \text{and}\ x_{(0)} = y_{(1)} = |z_{(0)}| = |z_{(1)}|=2 \ \text{with} \ z_{(0)}z_{(1)}>0; \label{rj_1} 
\end{equation}
or
\begin{equation}
     m \equiv n \equiv 1 \ (\bmod\ 2) \ \text{and}\ x_{(0)} = y_{(1)} = r_2, \ |z_{(0)}| = t, \ |z_{(1)}|=s_2 \ \text{with} \ z_{(0)}z_{(1)}>0.\label{rj_2}
\end{equation}

Since at least one of $D(4)$-quadruples is irregular, we can use the following lemma for a lower bound for the element $b$.

\begin{lemma}[{\cite[Lemma 2.2]{mbt}}]
    Let $\{a,b,c,d\}$ be an irregular $D(4)$-quadruple with $a<b<c<d$. Then $b>10^5$.
\end{lemma}

Also, note that it holds $a_2-a_1 \geq 2$ since  \cite[Theorem 1.1]{mbtap1} assures that $a_2 \neq a_1 + 1$ for all $a_1 \neq 3$. If $a_1 = 3$ and $a_2=a_1+1=4$, then $\{a_1,a_2,b,c,d\}$ would be a $D(4)$-quintuple, which is a contradiction to \cite[Theorem 1]{nas2}. Hence, as $a_1 \geq 1$, we have $a_2 \geq 3$.

\smallskip

We will now recall some useful tools, first introduced in \cite{rickert}, and establish some improvements on bounds for indices $m$ and $n$ in terms of elements of an irregular $D(4)$-quadruple.

\begin{theorem}\label{tm: gornje_ograde}
Let $a_1,a_2$ be integers such that $0<a_1 \leq a_2-2,\ a_2 \geq 3$ and let's denote $a_1' := \max\{4(a_2-a_1),4a_1\}$. Also, let $N$ be a multiple of $a_1a_2$ and assume that $N \geq 396.59a_1'a_2^2(a_2-a_1)^2$. Then the numbers $\theta_1 = \sqrt{1+\frac{4a_2}{N}}$ and $\theta_2 = \sqrt{1+\frac{4a_1}{N}}$ satisfy
$$ \left\{ \left| \theta_1-\frac{p_1}{q} \right| , \left|\theta_1-\frac{p_2}{q} \right|     \right\} > \frac{a_1}{3.53081 \cdot 10^{27} a_1'a_2N}q^{-\lambda}$$
for all integers $p_1,p_2,q$, with $q>0$, where 
\[\lambda=1+\frac{\log(2.500788a_1^{-1}a_1'a_2N)}{\log(0.04216a_1^{-1}a_2^{-1}(a_2-a_1)^{-2}N^2)}<2.\]
\end{theorem}
\begin{proof}
The validity of this assertion is obvious from the proof of \cite[Theorem 2]{nas2} with $g$ replaced by $1$.
\end{proof}
The proof of the following result can be derived  as in \cite[Lemma 6]{dujram}.
\begin{lemma}\label{gornje_ograde2}
Let $N=a_1a_2c$ and let $\theta_1$, $\theta_2$ be as in Theorem \ref{tm: gornje_ograde}. Then, all positive solutions to the system of Pellian equations (\ref{prva}) and (\ref{druga}) satisfy
$$\max\left\{\left|\theta_1-\frac{s_1a_2x_1}{a_1a_2z}\right|,\left|\theta_2-\frac{s_2a_1x_2}{a_1a_2z}\right|\right\}<\frac{2c}{a_1}z^{-2}.$$
\end{lemma}

The next result is proven similarly as \cite[Lemma 6]{nas2} or \cite[Lemma 8]{sestorka}.
\begin{lemma}\label{gornje_ograde3}
     Assume that $c >4.1\cdot10^{-5}a_2b^3$ and $b>10^5$. If $z=w_n$ with $n \geq 4$, then 
    \[ \log z > \frac{n}{2}\log(bc) .\]
\end{lemma}  
Combining these results yields a very useful lemma.
\begin{lemma}\label{lemma: gornja_n_1}
    Let $\{a_1,b,c,d\}$ and $\{a_2,b,c,d\}$ be $D(4)$-quadruples such that the assumptions in Theorem \ref{tm: gornje_ograde} and Lemma \ref{gornje_ograde3} hold. If $z=v_m=w_n$ has a solution for some integers $m$ and $n$ with $n \geq 4$, then 
   $$ n < \frac{8\log(8.4034 \cdot 10^{13} a_1^{1/2}(a_1')^{1/2}a_2^2c)\log(0.20533a_1^{1/2}a_2^{1/2}(a_2-a_1)^{-1}c)}{\log(bc)\log(0.01685a_1(a_1')^{-1}a_2^{-1}(a_2-a_1)^{-2}c)}. $$
\end{lemma}
\begin{proof}
    We apply Theorem \ref{tm: gornje_ograde} with
    \[q=a_1a_2z,\:\: p_1=s_1a_2x_1,\:\: p_2=s_2a_1x_2,\:\: N=a_1a_2c.\]
    Combining Theorem \ref{tm: gornje_ograde} with Lemma \ref{gornje_ograde2}, we get
    \[ z^{2-\lambda} < 7.06162 \cdot 10^{27} a_1a_1'a_2^4c^2 < (8.4034 \cdot 10^{13} a_1^{1/2}(a_1')^{1/2}a_2^2c)^2  \]
    and
    $$\frac{1}{2-\lambda} < \frac{2\log(0.20533a_1^{1/2}a_2^{1/2}(a_2-a_1)^{-1}c)}{\log(0.01685a_1(a_1')^{-1}a_2^{-1}(a_2-a_1)^{-2}c)}.$$
    Our claim now results from the comparison of the above inequalities with those in Lemma \ref{gornje_ograde3}.
\end{proof}

From \cite[Lemma 2.6]{mbt} we know that a regular quadruple is obtained as a solution of $z=v_m=w_n$ for $m,n\leq 2$. Moreover, \cite[Lemma 2.10]{mbt} together with \cite[Theorem 1.3]{mbt} gives us the next lower bound for $n$ and $m$.
\begin{lemma}\label{lema: n vece 7}
If $v_m=w_n$ has a solution for $m>2$, then $6\leq m \leq  2n + 1$ and $n \geq 7$.
\end{lemma}

\begin{lemma}[{\cite[Lemma 2.9]{mbt}}]\label{lemma: epsilon}
    Assume that $z=v_m=w_n$ has a solution for some integers $m$ and $n$. If $c>b^{\varepsilon}$, $b>10^5$, $1 \leq \varepsilon < 12$, then 
    $$ m < \frac{\varepsilon+1}{0.999\varepsilon}n + 1.5 - 0.4\frac{\varepsilon+1}{0.999\epsilon}.$$
\end{lemma}
The next lemma gives us a lower bound for index $n$ in terms of elements $b$ and $c$ and will be used throughout the paper in combination with Lemma \ref{lemma: gornja_n_1}. 
\begin{lemma}\label{lema: donja na n u terminima b i c}
    Assume that $c>\max\{0.56b^3,0.001a_2^2b^2\}$ and $b>10^5$. If $z=v_m=w_n$ has a solution for some integers $m$ and $n$ with $n\geq 4$, then $m\equiv n\ (\bmod \ 2)$ and $n>0.09226b^{-1/2}c^{1/4}$ for odd $n$ and $n>0.340134b^{-1/2}c^{1/2}$ for even $n$.
\end{lemma}

\begin{proof}
    As in the proof of \cite[Lemma 3.1]{mbt} we know that $m\equiv n \ (\bmod \ 2)$. 

    First, if $m$ and $n$ are even, we consider the case $z_0=z_1=\pm 2$. Since $c>0.56b^3>b^{2.94}$, we have $\varepsilon=2.94$ in Lemma \ref{lemma: epsilon} which yields
    $m<1.3415n+0.97.$
From Lemma \ref{lema: n vece 7} we have $n\geq 8$ so we can take
$m<1.47n.$
These observations are applied to \cite[Lemma 4.1]{mbt} with parameters $L=1.47$, $\rho=1$, $a_0=3$, $b_0=10^5$, $c_0=0.56\cdot10^{15}$ (Note that Lemma requires $\rho>1$, but the same holds for $\rho=1$). We get 
$m>0.499997b^{-1/2}c^{1/2},$
hence
$$n>0.340134b^{-1/2}c^{1/2}.$$

Second, we consider the case when $m$ and $n$ are odd. From Lemma \ref{lema: n vece 7} we have $n\leq m$ and $n\geq 7$ so we can, similarly as before, take
$m<1.48n.$

From the proof of \cite[Lemma 13]{fil_xy4} we know that the congruences 
\begin{align}
    \pm 2t(a_2 m(m+1)-bn(n+1))& \equiv 2r_2s_2(n-m)\ (\bmod \ c),\label{cong: prva_nep}\\
    \pm 2s_2(a_2 m(m+1)-bn(n+1))& \equiv 2r_2t(n-m)\ (\bmod \ c)\nonumber,
\end{align}
hold. Since $s_2t\equiv 16\ (\bmod \ c)$ after multiplying both sides of these congruences we see that we have
\begin{equation}\label{eqn:congruence_neparni}
    64(a_2m(m+1)-bn(n+1))^2\equiv 64r_2^2(m-n)^2\ (\bmod \ c).
\end{equation}
Assume to the contrary, that $n\leq 0.09226b^{-1/2}c^{1/4}$. 

If we show that 
$$|a_2m(m+1)-bn(n+1)|+r_2(m-n)<\frac{c^{1/2}}{64},$$
then we know that (\ref{eqn:congruence_neparni}) is an equality.  First, let us assume that $bn(n +1)\geq a_2m(m+1)$. Then, since $r_2<b$ we have
\begin{align*}
    bn(n+1)-&a_2m(m+1)+r_2(m-n)\\
    &<bn^2+bn+0.48r_2n-a_2m(m+1)\\
    &<bn^2+0.96bn<1.5bn^2.
\end{align*}
And since we assumed $n\leq 0.09226b^{-1/2}c^{1/4}$ we have $1.5n^2<0.0128c^{1/2}<1/64\cdot c^{1/2}$. 

Now, assume that $bn(n +1)< a_2m(m+1)$. Then
\begin{align*}
    a_2m(m+1)-&bn(n+1)+r_2(m-n)\\
    &<1.48a_2n(1.48n+1)-bn(n+1)+r_2\cdot 0.48 n\\
    &<n^2(2.1904a_2-b)+n(1.48a_2-b)+0.481n\sqrt{a_2 b}.
\end{align*}
Since $a_2<b<c^{1/3}$ and $0.56\cdot10^{15}<c$ and $n\leq 0.09226b^{-1/2}c^{1/4}$ we further have,
\begin{align*}
    a_2m(m+1)-&bn(n+1)+r_2(m-n)<\\
    &<(1.1904\cdot 0.09226^2+0.961\cdot 0.09226\cdot 10^{-15/12}\cdot 0.56^{-1/6})c^{1/2}\\
    &<1/64c^{1/2}.
\end{align*}
So, congruence (\ref{eqn:congruence_neparni}) is an equality, which yields
\begin{equation}\label{eqn: jdkost_iz_kong}
    |a_2m(m+1)-bn(n+1)|=r_2(m-n).
\end{equation}
After inserting that in congruence (\ref{cong: prva_nep}) we get
\begin{equation}\label{eqn: cong_2}
2r_2(m-n)(\pm t+s_2)\equiv 0\ (\bmod \ c).
\end{equation}
Since $n\leq 0.09226b^{-1/2}c^{1/4}$ and $c^{1/4}>0.1778\sqrt{ab}$, inequality
\begin{align*}
    2r_2(m-n)|\pm t+s_2|&<0.96nr_2\cdot 2t\\
    &<1.92n\cdot 1.0002\sqrt{a_2 b}\cdot 1.00001\sqrt{bc}<c,
\end{align*}
asserts that (\ref{eqn: cong_2}) is also an equality:
$$2r_2(m-n)(\pm t+s_2)= 0.$$
So, either $t=s_2$ or $m=n$. From $t=s_2$, we would have $b=a_2$, which cannot hold. If $m=n$, from (\ref{eqn: jdkost_iz_kong}) we would also get $a_2=b$, a contradiction. 

This implies that our assumption was wrong and $n>0.09226b^{-1/2}c^{1/4}$. 
\end{proof}

Let us consider the first step of constructing $D(4)$-pairs $\{a_1,b\}$ and $\{a_2,b\}$ when $a_1$ and $a_2$ are given integers. We have that $b=(r_1^2-4)/a_1$, where $r_1$ is a solution to the equation 
$$r_1^2a_2-r_2^2a_1=4(a_2-a_1),$$
with unknowns $(r_1,r_2)$. We can rewrite this equation and observe
\begin{equation}\label{eq:pellova_konacno}
(r_1a_2)^2-r_2^2(a_1a_2)=4a_2(a_2-a_1).
\end{equation}
If $a_1a_2$ is a perfect square, this is the equation of a form
\begin{equation}\label{eq:razlika_kv}(r_1a_2)^2-(r_2\sqrt{a_1a_2})^2=4a_2(a_2-a_1)\end{equation}
which has only finitely many solutions for each pair $\{a_1,a_2\}$. 

On the other hand, if $a_1a_2$ isn't a perfect square, (\ref{eq:pellova_konacno}) is a Pellian equation whose solutions can be easily described for any fixed pair $\{a_1,a_2\}$. 

If we are extending $D(4)$-pairs $\{a_1,b\}$ and $\{a_2,b\}$ to $D(4)$-triples $\{a_1,b,c\}$ and $\{a_2,b,c\}$ we have a very similar problem to solve. The element $c$ is equal to $(s_1^2-4)/a_1$ where $s_1$ is obtained from equation (\ref{eq:pellova_konacno}) or (\ref{eq:razlika_kv}) in which $(r_1,r_2)$ is replaced by $(s_1,s_2)$. Thus, both $b$ and $c$ are elements of the solution sequences of the Pellian equation (\ref{eq:pellova_konacno}) or elements of the set of finitely many solutions of equation (\ref{eq:razlika_kv}).

In most of our proofs, we'll be dealing with finitely many remaining pairs $\{a_1,a_2\}$ with known upper (and lower) bounds for the elements $b$ and $c$. This means to complete the proof, it will be sufficient to search for elements of solution sequences that also satisfy some given bounds.

\section{Note on applying Lemma \ref{lemma: gornja_n_1}}\label{sekcija_treca}

Let $\{a_1,b,c,d\}$ and $\{a_2,b,c,d\}$ be such $D(4)$-quadruples that the assumptions made in Theorem \ref{tm: gornje_ograde} and Lemma \ref{gornje_ograde3} hold, i.e. that Lemma \ref{lemma: gornja_n_1} holds. Let us assume that $a_1,a_2$ and $b$ are fixed values and consider the function
$$
\varphi(c)=
\frac{8\log(8.4034 \cdot 10^{13} a_1^{1/2}(a_1')^{1/2}a_2^2c)\log(0.20533a_1^{1/2}a_2^{1/2}(a_2-a_1)^{-1}c)}{\log(bc)\log(0.01685a_1(a_1')^{-1}a_2^{-1}(a_2-a_1)^{-2}c)}.$$
In many cases, when using Lemma \ref{lemma: gornja_n_1}, it is useful to demonstrate that this function decreases with respect to
 $c$  for the observed values of 
 $c$. This allows us to obtain an upper bound for $\varphi(c)$ when we replace $c$ with its lower bound (usually the lower bound is expressed in terms of $a_1$, $a_2$ and/or $b$).

Since $a_2>a_1\geq 1$, it can easily be seen that $$0.20533a_1^{1/2}a_2^{1/2}(a_2-a_1)^{-1}>0.01685a_1(a_1')^{-1}a_2^{-1}(a_2-a_1)^{-2}$$ so
$$\varphi_1(c)=\frac{\log(0.20533a_1^{1/2}a_2^{1/2}(a_2-a_1)^{-1}c)}{\log(0.01685a_1(a_1')^{-1}a_2^{-1}(a_2-a_1)^{-2}c)}>1$$
and $\varphi_1(c)$ is decreasing in $c$. 

Our goal is to reach the same conclusion regarding
$$\varphi_2(c)=\frac{\log(8.4034 \cdot 10^{13} a_1^{1/2}(a_1')^{1/2}a_2^2c)}{\log(bc)},$$
i.e. we need to show that $b<8.4034 \cdot 10^{13}a_1^{1/2}(a_1')^{1/2}a_2^2$. Hence, if both $\varphi_1(c)$ and $\varphi_2(c)$ are decreasing and both $\varphi_1(c)>0$ and $\varphi_2(c)>0$ it implies that $\varphi(c)$ is decreasing in $c$. 

Since \cite[Theorem 1.2]{mbtap1} states that $c>0.25b^3$ and $0.25b^3>0.25\cdot 10^5a_2^2$ it is easy to see that all arguments of logarithms in the expression for $\varphi(c)$ are greater than $1$ except maybe $0.01685a_1(a_1')^{-1}a_2^{-1}(a_2-a_1)^{-2}c$. So, we must check that 
\begin{equation}\label{ineq:provjera za c}
c>(0.01685a_1(a_1')^{-1}a_2^{-1}(a_2-a_1)^{-2})^{-1}=0.01685^{-1}a_1^{-1}a_1'a_2(a_2-a_1)^2
\end{equation}
holds to conclude that $\varphi(c)$ is decreasing in $c$.

Let us first prove that $b\geq 8.4034 \cdot 10^{13}a_1^{1/2}(a_1')^{1/2}a_2^2$ cannot hold. 

\begin{lemma}\label{lemma: gornja an b}
     Assume that $\{a_1,b,c,d\}$ and $\{a_2,b,c,d\}$ are $D(4)$-quadruples with $a_1<a_2<b<c<d$. Then 
     $$b< 8.4034 \cdot 10^{13}a_1^{1/2}(a_1')^{1/2}a_2^2.$$
\end{lemma}
\begin{proof}
Let us assume to the contrary,  
$b>8.4034 \cdot 10^{13}a_1^{1/2}(a_1')^{1/2}a_2^2.$ 
Since $c>0.25a_1^2b^3$, we also have
\begin{equation}c>1.186\cdot 10^{42}a_1^5a_2^6\label{jdba_donja_c}
\end{equation}
for both options $a_1'\in\{4a_1,4(a_2-a_1)\}$. Observe that 
\begin{align*}8.4034\cdot 10^{13}a_1^{1/2}(a_1')^{1/2}a_2^2&<1.681\cdot 10^{14}\cdot a_2^3,\\
0.20533a_1^{1/2}a_2^{1/2}(a_2-a_1)^{-1}&<0.102665\cdot a_2,\\
0.01685a_1(a_1')^{-1}a_2^{-1}(a_2-a_1)^{-2}&>0.0042125\cdot a_2^{-4}.\end{align*}
Now, from Lemma \ref{lemma: gornja_n_1} we have
$$n<\frac{\log(1.681\cdot 10^{14}\cdot a_2^3c)\log(0.102665\cdot a_2c)}{\log(2.376\cdot 10^{14}\cdot a_2^2c)\log(0.0042125\cdot a_2^{-4}c)}.$$
The right-hand side of this inequality is decreasing in $c$, so we can use (\ref{jdba_donja_c}), $n>0.09226(c/b^2)^{1/4}$ and $c/b^2>0.25a_1^2b\geq 0.25b>4.2\cdot 10^{13}\cdot a_2^2$ in the previous inequality. We get that it cannot hold for $a_2>1$.
\end{proof}

We have the following corollary from these observations and Lemma \ref{lemma: gornja an b}.
\begin{cor}\label{cor: primjena leme 3.4}
   Let $\{a_1,b,c,d\}$ and $\{a_2,b,c,d\}$ be $D(4)$-quadruples such that the assumptions of Lemma \ref{lemma: gornja_n_1} hold. If $c>0.01685^{-1}a_1^{-1}a_1'a_2(a_2-a_1)^2$, then the function $\varphi(c)$ is decreasing in $c$.
\end{cor}

\section{Upper bound for $c$}\label{sekcija_upper_c}

The goal of this section is to prove an upper bound for $c$ in terms of $a_2$ and $b$ and to prove one lower bound for $a_2$ in terms of $a_1$.
\begin{theorem}\label{tm5.1}
    Assume that $\{a_1,b,c,d\}$ and $\{a_2,b,c,d\}$ are $D(4)$-quadruples with $a_1<a_2<b<c<d$. Then the following hold:
    \begin{enumerate}[1)]
        \item $a_2>2a_1$;
        \item $0.25a_1^2b^3 < c < 0.25a_2^2b^3$.
    \end{enumerate}
\end{theorem}

We begin by improving \cite[Theorem 1.2]{mbtap1}.

\begin{theorem}  \label{tm5.2}
    If $\{a_1,b,c\}$ and $\{a_2,b,c\}$ are $D(4)$-triples with $a_1<a_2<b<c\leq 0.25a_1^2b^3$, then $\{a_1,a_2,b,c\}$ is a $D(4)$-quadruple.
\end{theorem}
\begin{proof}
   This statement and the proof are the same as \cite[Theorem 1.2]{mbtap1} with the assumption $c\leq 0.25b^3$ replaced by $c \leq 0.25a_1^2b^3$. 
\end{proof}

Since \cite[Theorem 1]{nas2} holds, Theorem \ref{tm5.2} has the following corollary, which proves the first inequality of the second statement in Theorem \ref{tm5.1}.

\begin{cor}\label{cor: donja na c i d}
If $\{a_1,b,c,d\}$ and $\{a_2,b,c,d\}$ are $D(4)$-quadruples with $a_1<a_2<b<c<d$, then $c>0.25a_1^2b^3$ and $d>0.25a_1^2c^3$.
\end{cor}
\begin{proof}
    First, assume that $c \leq 0.25a_1^2b^3$. By using the previous theorem on $D(4)$-triples $\{a_1,b,c\}$ and $\{a_2,b,c\}$  we get that $\{a_1,a_2,b,c\}$ is a $D(4)$-quadruple, so $\{a_1,a_2,b,c,d\}$ is a $D(4)$-quintuple, which cannot hold. For the second statement, assume that $d \leq 0.25a_1^2c^3$.\ By using the previous theorem on $\{a_1,c,d\}$ and $\{a_2,c,d\}$ we get that $\{a_1,a_2,c,d  \}$ is a $D(4)$-quadruple, so $\{a_1,a_2,b,c,d \}$ is a $D(4)$-quintuple, which is a contradiction.
\end{proof}

The proof of Corollary \ref{cor: nereg} now easily follows.

\begin{proof}[Proof of Corollary \ref{cor: nereg}]
Notice that since $a_ibc+4c>r_is_it$ we have $d_+(a_i,b,c) < a_ibc + 4c$. On the other hand, since $c>0.25a_1^2b^3$, we have 
$$d>0.25a_1^2c^3>0.25a_1^2(0.25a_1^2b^3)^2c>0.25^3a_1^6b^4a_2bc.$$
So, by using the fact that $b>10^5$ it is easy to see that $d>d_+(a_2,b,c)>d_+(a_1,b,c)$.
\end{proof}

\begin{proof}[Proof of Theorem \ref{tm5.1}]
\emph{1)} Assume that $a_2\leq 2a_1$.
By Corollary \ref{cor: donja na c i d} we know $c>0.25a_1^2b^3$. We can use Theorem \ref{tm: gornje_ograde} since $a_1'=4a_1$, $b>10^5$, $b>a_2>a_2-a_1\geq 2$, hence 
$$N=a_1a_2c>\frac{0.25}{8}10^5a_1'a_2^2(a_2-a_1)^2>396.59a_1'a_2^2(a_2-a_1)^2.$$
Also, it is easy to see that lower bounds for $c$ from the statements of Lemmas \ref{gornje_ograde3}  and \ref{lema: donja na n u terminima b i c} are satisfied since $a_2-a_1\geq 2$ and $a_2\leq 2a_1$ imply $a_1\geq 2$.

Hence, we can combine the lower bound $n>0.09226b^{-1/2}c^{1/4}$ and the upper bound for $n$ from Lemma \ref{lemma: gornja_n_1}.

 Note again that $a_1<a_2\leq 2a_1$ implies $a_1'=\max\{4a_1,4(a_2-a_1)\}=4a_1$. Also, $2\leq a_2-a_1$. So we have
\begin{align*}
 8.4034\cdot 10^{13}a_1^{1/2}(a_1')^{1/2}a_2^2c&<1.68068\cdot 10^{14}a_2^3c,\\
 0.20533a_1^{1/2}a_2^{1/2}(a_2-a_1)^{-1}c&<0.102665a_2c,\\
 0.01685a_1(a_1')^{-1}a_2^{-1}(a_2-a_1)^{-2}c&>0.01685a_2^{-3}c.
\end{align*}
If we moreover have $b>ka_2$ for some $k\geq1$, then
\begin{equation}\label{eq:prva_njdnok_51}
    n<\frac{8\log(1.68068\cdot 10^{14}a_2^3c)\log(0.102665a_2c)}{\log(ka_2c)\log(0.01685a_2^{-3}c)}.
\end{equation}
The right-hand side of the inequality (\ref{eq:prva_njdnok_51}) is decreasing in $c$ when $c>238a_2^3$ which is true in our case. Since $$c>0.25a_1^2b^3>0.0625\cdot k^3a_2^5$$
and $c/b^2>0.0625a_2^2b>6.25\cdot 10^3a_2^2$ we observe the inequality
\begin{equation}\label{eq:druga_njdnok_51}
    0.09226\cdot (6.25\cdot 10^3a_2^2)^{1/4}<\frac{8\log(1.050425\cdot 10^{13}a_2^8k^3)\log(0.006417a_2^6k^3)}{\log(0.0625a_2^6k^4)\log(0.001053125a_2^2k^3)}.
\end{equation}
We have $b>a_2$, so $k=1$ and we get from (\ref{eq:druga_njdnok_51}) that $a_2\leq 7401$. Now, $b>10^5\geq \frac{10^5}{7401}a_2$, so we can take $k=\frac{10^5}{7401}$ and calculate that $a_2\leq2860$. We repeat the process six more times and get $a_2\leq 1976$.

Notice that $(0.01685a_1(a_1')^{-1}a_2^{-1}(a_2-a_1)^{-2})^{-1}<10^{15}<c$ so we can use Corollary \ref{cor: primjena leme 3.4} for any pair $\{a_1,a_2\}$, $a_1<a_2<1976$. For these finitely many pairs $\{a_1,a_2\}$, 
we consider a program in Wolfram Mathematica where we observed the original inequality from Lemma \ref{lemma: gornja_n_1}, and substituted $c$ with its lower bound $0.25a_1^2b^3$. 
For all values $\{a_1,a_2\}$ with $1406\leq a_2\leq 1976$ this inequality yields a contradiction to $b>10^5.$
For each of the remaining pairs $\{a_1,a_2\}$, $a_2\leq 1405$ we searched for solutions as described at the end of Section \ref{sekcija_druga}.
To check the pairs where $a_1a_2$ isn't a perfect square, we used the inequality from Lemma \ref{lemma: gornja_n_1} to obtain an upper bound for $b$ and then searched for solutions of the equation (\ref{eq:pellova_konacno}) where $b > 10^5$ and $b$ is less than the calculated bound. It remained to search for solutions of equation (\ref{eq:pellova_konacno}) which yield the element $c$ such that $0.25a_1^2b^3< c< 39247b^4$ and $\sqrt{bc+4}$ is an integer. This search returned that none of such elements $b$ and $c$ exist within those bounds.

\medskip 
\emph{2)} Assume to the contrary, that $c\geq 0.25a_2^2b^3$. It is easy to see that the assumptions of Lemmas \ref{lemma: gornja_n_1} and \ref{lema: donja na n u terminima b i c} are satisfied.
We know from 1) that $a_2>2a_1$, so $a_1'=4(a_2-a_1).$ Also, $\frac{a_2}{2}<a_2-a_1<a_2$. Then 
\begin{align*}
 8.4034\cdot 10^{13}a_1^{1/2}(a_1')^{1/2}a_2^2c&<1.1885\cdot 10^{14}a_2^3c\\
 0.20533a_1^{1/2}a_2^{1/2}(a_2-a_1)^{-1}c&<0.29039c\\
 0.01685a_1(a_1')^{-1}a_2^{-1}(a_2-a_1)^{-2}c&>0.0042125a_2^{-4}c.
\end{align*}
If we have $b> ka_2$ for some $k\geq 1$, then $c>0.25a_2^5k^3$ and $c/b^2\geq 0.25\cdot 10^5a_2^2$. Since $0.01685^{-1}\cdot 4\cdot a_1^{-1} a_2(a_2-a_1)^3<238a_2^4$ and $c>0.25bk^2a_2^4>25000a_2^4$, Corollary \ref{cor: primjena leme 3.4} implies that $\varphi(c)$ is decreasing in $c$. Hence, we use $c>0.25a_2^5k^3$ and observe the inequality 
\begin{equation}\label{eq:treca_njdnok_51}
    0.09226(0.25\cdot 10^5a_2^2)^{1/4}<\frac{8\log(6.2872\cdot 10^{16}a_2^8k^3)\log(153.61632a_2^5k^3)}{\log(0.25a_2^6k^4)\log(2.2284 a_2k^3)}.
\end{equation}
We have $b>a_2$.  So we can take $k=1$ in (\ref{eq:treca_njdnok_51}) which yields $a_2\leq 32499$. 
Now, $k=10^5/32499$ implies $a_2\leq 10741$, and after ten more iterations we get $a_2 \leq 2050$.

For those finitely many remaining pairs $\{a_1,a_2\}$ we repeat the process described in the first case of this proof. Again, it leads to the conclusion that there are no quadruples $\{a_1,a_2,b,c\}$ which satisfy all necessary conditions.\hfill\qedhere

\end{proof}

\begin{proof}[Proof of Corollary \ref{cor: najvisedva}] Assume that there exist three positive integers $a_1, a_2, a_3$ with $a_1 < a_2 < a_3 < \min\{b,c,d\}$ such that $\{a_i,b,c,d\}$ ($i \in \{1,2,3\}$) are $D(4)$-quadruples. Applying Theorem \ref{tm5.1} to $\{a_1,b,c,d\}$ and $\{a_2,b,c,d\}$ yields $c<0.25a_2^2b^3$. Also, applying that same theorem to $\{a_2,b,c,d\}$ and $\{a_3,b,c,d\}$ yields $0.25a_2^2b^3 < c$, a contradiction.
\end{proof}

\section{Lower bounds for $a_2$}\label{sekcija_bound_1}

Proving the properties of the case where $a_1=1$ is sometimes difficult since it does not satisfy the conditions of Lemma \ref{lema: donja na n u terminima b i c}, thus requiring a different approach. For the proof of Proposition \ref{prop: prop1} we will first show separately that $a_1=1$ implies $a_2\geq 5$. 

\begin{lemma}\label{lem:a1je1}
    Assume that $\{1,b,c,d\}$ is a $D(4)$-quadruple. Then $\{3,b,c,d\}$ and $\{4,b,c,d\}$ are not $D(4)$-quadruples.  
\end{lemma}
\begin{proof}
    Assume that $\{1,b,c,d\}$ and $\{3,b,c,d\}$ are $D(4)$-quadruples. Then $r_2^2-3r_1^2=-8$ and $s_2^2-3s_1^2=-8.$ We see that both $b$ and $c$ can be derived from a sequence of solutions of equation \begin{equation}\label{eq:pellova_r_13}
    x^2-3y^2=-8,
    \end{equation}
    since $b=r_1^2-4=y^2-4$ and $c=s_1^2-4=y^2-4$. 

    It is easy to prove that positive integer $y$ is given by
    $$
    y=y_n:=\frac{3+\sqrt{3}}{3}(2+\sqrt{3})^n+\frac{3-\sqrt{3}}{3}(2-\sqrt{3})^n,\ n\geq 0,
    $$
   where  $n$ is a nonnegative integer.

    From Theorem \ref{tm5.1} we have $0.25b^3<c<0.25\cdot 9b^3=2.25b^3$. Let us denote $b=y_j^2-4$ and $c=y_k^2-4$ for some integers $j,k>0$. Then 
    \begin{align*}
        y_k^2&<2.25(y_j^2-4)^3+4<2.25y_j^6,\\
        y_k^2&>0.25(y_j^2-4)^3+4>0.24y_j^6.
    \end{align*}

    For $j=1$ we have $b=32$, and after observing values of $c=y_k^2-1$ for $k=2,3,4$ we conclude that none of them satisfies $0.25b^3<c<0.25\cdot 9b^3$.  
     The same is true for $b=y_j^2-4$, $j\leq 3$. So we can assume from now on that $k>j>3$.  
    Since $1.57735(2+\sqrt{3})^n<y_n<1.57736(2+\sqrt{3})^n$ for $n\geq 6$, inequality
    $$1.57735^2(2+\sqrt{3})^{2k}<2.25\cdot 1.57736^6(2+\sqrt{3})^{6j},$$
     implies $k\leq 3j+1$. On the other hand, from 
    $$0.24\cdot 1.57735^6(2+\sqrt{3})^{6j}<1.57736^2(2+\sqrt{3})^{2k},$$
   we get $k>3j+0.15$, i.e. conclude $k=3j+1$.
    But, 
    \begin{align*}
&c-0.25\cdot3^2b^3=y_{3j+1}^2-4-2.25(y_j^2-4)^3=\\
&=8 ((2 + \sqrt{3})^{4 n + 2} + (2 - \sqrt{3})^{4 n + 2}) - 
 34 ((2 + \sqrt{3})^{2 n + 1} + (2 - \sqrt{3})^{2 n + 1}) + 56>0
    \end{align*}
  
    which is a contradiction to Theorem \ref{tm5.1}. So, $\{3,b,c,d\}$ cannot be a $D(4)$-quadruple if $\{1,b,c,d\}$ is.

    Now, assume that $\{1,b,c,d\}$ and $\{4,b,c,d\}$ are $D(4)$-quadruples. Then there exist positive integers $r_1,r_2$ such that 
    $b+4=r_1^2$ and $4b+4=r_2^2.$
   It follows that $r_1$ and $r_2$ satisfy the equation
 \begin{equation*}
 (2r_1)^2-r_2^2=12,
 \end{equation*}
 which implies $r_1=r_2=2$. So there are no such $D(4)$-quadruples.
\end{proof}
We are now ready to prove statement 1) of Theorem \ref{glavni}.
\begin{proposition} 
\label{prop: prop1}
    Assume that $\{a_1,b,c,d\}$ and $\{a_2,b,c,d\}$ are $D(4)$-quadruples with $a_1<a_2<b<c<d$. Then $a_2>4a_1$.
\end{proposition}
\begin{proof}
The statement for $a_1=1$ follows from Lemma \ref{lem:a1je1}. So, we may assume that $a_1\geq 2$. Also, let us assume to the contrary, that $2a_1 < a_2 \leq 4a_1$. Since $a_1\geq 2$, it is easy to check that the assumptions of Lemmas \ref{lemma: gornja_n_1} and \ref{lema: donja na n u terminima b i c} are satisfied. As we know $a_2 > 2a_1$, it follows that $a_1'=4(a_2-a_1)$. Also, from $a_2 \leq 4a_1$ we obtain $a_2-a_1 \leq \frac{3}{4}a_2$. Then
    \begin{align*}
 8.4034\cdot 10^{13}a_1^{1/2}(a_1')^{1/2}a_2^2c&<1.02921\cdot 10^{14}a_2^3c\\
 0.20533a_1^{1/2}a_2^{1/2}(a_2-a_1)^{-1}c&<0.29039c\\
 0.01685a_1(a_1')^{-1}a_2^{-1}(a_2-a_1)^{-2}c&>0.002496a_2^{-3}c.
\end{align*}
On the other hand, if we write $b>ka_2$ for some $k \geq 1$, we have $c > 0.015625a_2^5k^3$ and $c/b^2 \geq 1.5626 \cdot 10^3 a_2^2$. Since $0.01685^{-1}\cdot 4\cdot a_2(a_2-a_1)^3<238a_2^4$ and $c>0.25bk^2a_2^4>25000a_2^4$, Corollary \ref{cor: primjena leme 3.4} implies that $\varphi(c)$ is decreasing in $c$. So we can use lower bound $c > 0.015625a_2^5k^3$ and observe 
\begin{equation}\label{eq:cetvrta_njdnok_51}
    0.09226(1.5625\cdot 10^3a_2^2)^{1/4}<\frac{8\log(1.02921\cdot 10^{14}\cdot\frac{1}{64}a_2^8k^3)\log(0.29039\cdot\frac{1}{64}a_2^5k^3)}{\log(\frac{1}{64}a_2^6k^4)\log(0.002496 \cdot \frac{1}{64} a_2^2k^3)}.
\end{equation}
We have $b > a_2$. So we can take $k = 1$ in (\ref{eq:cetvrta_njdnok_51}) which yields $a_2\leq 15917$. Now, $k=10^5/15917$ implies $a_2 \leq 7478$, and after eight more iterations we get $a_2 \leq 5179$.

For those finitely many remaining pairs $\{a_1,a_2\}$ we repeat the process described in the proof of Theorem \ref{tm5.1} and conclude that there are no quadruples $\{a_1,a_2,b,c\}$ which satisfy all necessary conditions.
\end{proof}
Unfortunately, some computer calculations took more than 24 hours to complete. To reduce the remaining number of cases that need to be checked, we will divide our next proof into several steps. First, we will establish a weaker bound, and then we will use this bound to prove the final result.

\begin{theorem}
    Assume that $\{a_1,b,c,d\}$ and $\{a_2,b,c,d\}$ are $D(4)$-quadruples with $a_1<a_2<b<c<d$. Then 
       $$a_2>0.0251a_1^2.$$
   
\end{theorem}

\begin{proof}    
We will first prove that $a_2>0.01a_1^2$. 
Assume to the contrary, that $a_2\leq 0.01a_1^2$. From Proposition \ref{prop: prop1} we have $a_2>4a_1$. This implies $a_1\geq 401$ and $a_2\geq 1605.$ 
It is easy to see that the assumptions of Lemmas \ref{lemma: gornja_n_1} and \ref{lema: donja na n u terminima b i c} are now satisfied.

 Notice that $3a_1<a_2-a_1<a_2\leq0.01a_1^2$. We obtain the following inequalities: 
    \begin{align*}
 8.4034\cdot 10^{13}a_1^{1/2}(a_1')^{1/2}a_2^2c&<1.6807\cdot 10^9a_1^{11/2}c\\
 0.20533a_1^{1/2}a_2^{1/2}(a_2-a_1)^{-1}c&<0.00685a_1^{1/2}c\\
 0.01685a_1(a_1')^{-1}a_2^{-1}(a_2-a_1)^{-2}c&>421250a_1^{-7}c,
\end{align*}
so we observe an inequality
\begin{equation}\label{eq:njdn_manja001}
    n<\frac{8\log(1.6807\cdot 10^{9}\cdot a_1^{11/2}c)\log(0.00685a_1^{1/2}c)}{\log(bc)\log(421250a_1^{-7}c)}.
\end{equation}

We will separate the proof in cases depending on relations between $b$ and $a_1^2$. 

\textbf{Case 1:} $b \geq a_1^2$. 

Here we have $c>0.25a_1^8$. The right-hand side of inequality (\ref{eq:njdn_manja001}) is decreasing in $c$ for $c>421250^{-1}a_1^7$ which obviously holds. Also, $n>0.09226\cdot 0.25^{1/4}a_1$. Applying all those inequalities to (\ref{eq:njdn_manja001}) yields $a_1\leq 552$. 

Now we can rewrite the inequality (\ref{eq:njdn_manja001}) and together with  $n>0.2465c^{1/12}$ 
obtain the numerical upper bound for  $c$, i.e. $c<2.2666\cdot 10^{24}$. This implies $b<2.086\cdot 10^{8}\cdot a_1^{-2/3}$. We consider these finitely many pairs $\{a_1,a_2\}$, $a_2\leq 0.01a_1^2$ and a computer search for extensions within these bounds for $b$ and $c$ returns that there are no such numbers.

\textbf{Case 2:} $b<a_1^2$. 

Observe that, since $a_1<a_2/4<b/4$, $\frac{3}{4}a_2<a_2-a_1<a_2\leq 0.01a_1^2$ and $0.25b^4<c<0.25b^5$ we have
   \begin{align*}
 8.4034\cdot 10^{13}a_1^{1/2}(a_1')^{1/2}a_2^2c&<8.4034\cdot 10^{13}b^3c<2.101\cdot 10^{13}b^8\\
 0.20533a_1^{1/2}a_2^{1/2}(a_2-a_1)^{-1}c&<0.137c<0.03425b^5\\
 0.01685a_1(a_1')^{-1}a_2^{-1}(a_2-a_1)^{-2}c&>1.725\cdot 10^7b^{-7}c>4.3125\cdot 10^6b^{-3}. 
\end{align*}

Also, $n>0.09226b^{-1/2}c^{1/4}>0.09226\cdot 0.25^{1/4}b^{1/2}$. Applying this to inequalities from Lemmas \ref{lemma: gornja_n_1} and \ref{lema: donja na n u terminima b i c} yields $b\leq 162$ which is a contradiction. Thus, we have proved that $a_2>0.01a_1^2$.

Assume now that $a_2 \leq 0.0251a_1^2$. From Proposition \ref{prop: prop1} we have $a_2>4a_1$ and get $a_1\geq 160$ and $a_2 > 4\cdot 160=640$. It is easy to see that the assumptions of Lemmas \ref{lemma: gornja_n_1} and \ref{lema: donja na n u terminima b i c} are satisfied.

We follow the steps from the first part of the proof.
    Notice that $0.003a_1^2<a_2-a_1<a_2\leq0.0251a_1^2$ holds. 
so we obtain an inequality
\begin{equation}\label{eq:njdn_manja00251}
    n<\frac{8\log(1.678\cdot 10^{10}a_1^{11/2}c)\log(0.3286a_1^{-1/2}c)}{\log(bc)\log(10613a_1^{-7}c)}.
\end{equation}

As before, we consider two cases.

\textbf{Case 1:} $b \geq a_1^2$. Since the same argument holds as before, 
 we can use $c>10613^{-1}a_1^7$ and $n>0.09226\cdot 0.25^{1/4}a_1$. Applying all those inequalities yields $a_1\leq 682$. 

Now we can rewrite the inequality (\ref{eq:njdn_manja00251}) by inserting numerical bounds for $a_1$ and together with $n>0.2465c^{1/12}$ obtain $c<3.06\cdot 10^{24}$ and $b<2.205\cdot 10^{8}\cdot a_1^{-2/3}$. We consider the finitely many pairs $\{a_1,a_2\}$ which satisfy $ 0.01a_1^2<a_2\leq 0.0251a_1^2$ and search for extensions within the bounds for $b$ and $c$. A computer search as described before returns that there is no $c$ within the given bounds.

\textbf{Case 2:} $b<a_1^2$. 
Because of $b>10^5$, this assumption yields $a_1 \geq 317$.

Observe that, since $a_1<a_2/4<b/4$, $0.007b<0.007a_1^2<a_2-a_1\leq 0.0251a_1^2$ for $a_1\geq 317$, and $0.25b^4<c<0.25b^5$ we have
   \begin{align*}
 8.4034\cdot 10^{13}a_1^{1/2}(a_1')^{1/2}a_2^2c&<8.4034\cdot 10^{13}b^3c<2.101\cdot 10^{13}b^8\\
 0.20533a_1^{1/2}a_2^{1/2}(a_2-a_1)^{-1}c&<14.5785c<3.645b^5\\
 0.01685a_1(a_1')^{-1}a_2^{-1}(a_2-a_1)^{-2}c&>1.09\cdot 10^6b^{-7}c>272500b^{-3}. 
\end{align*}
Also, $n>0.09226b^{-1/2}c^{1/4}>0.09226\cdot 0.25^{1/4}b^{1/2}$. Applying this to inequalities from Lemmas \ref{lemma: gornja_n_1} and \ref{lema: donja na n u terminima b i c} yields $b\leq 64$ which is a contradiction. 
\end{proof}

\section{Upper bound for $b$ }\label{sekcija_bound_2}

Now we wish to establish further relations between elements $a_2$ and $b$. 

When $a_1=1$, we have $c>0.25a_1^2b^3=0.25b^3$ and that lower bound does not satisfy Lemma \ref{lema: donja na n u terminima b i c}, so we prove the next theorem only when $a_1\geq 2$.
\begin{theorem}
    Assume that $\{a_1,b,c,d\}$ and $\{a_2,b,c,d\}$ are $D(4)$-quadruples with $a_1<a_2<b<c<d$. 
        If $2\leq a_1$, then $b<a_2^2$.
\end{theorem}
\begin{proof}
 The proof can be split into two cases. In the first case, when $a_1\leq 159$, we observe that $a_2 > 4a_1$, which gives a better lower bound for $a_2$, and if $a_1\geq 160$, a better lower bound for $a_2$ is given by $a_2 > 0.0251a_1^2$.

\textit{ Case $a_1\leq 159$}:
Assume to the contrary, that $b\geq a_2^2$. Then $a_2\leq b^{1/2}$ and $a_1<0.25b^{1/2}$. Since $a_1\geq 2$ it is easy to check that the assumptions of Lemmas \ref{lemma: gornja_n_1} and \ref{lema: donja na n u terminima b i c} are satisfied. It holds
 \begin{align*}
 8.4034\cdot 10^{13}a_1^{1/2}(a_1')^{1/2}a_2^2c&<8.4034\cdot 10^{13}b^{3/2}c,\\
 0.20533a_1^{1/2}a_2^{1/2}(a_2-a_1)^{-1}c&<0.0153b^{1/2}c,\\
 0.01685a_1(a_1')^{-1}a_2^{-1}(a_2-a_1)^{-2}c&>0.0000115c, 
\end{align*}
where we have used $a_2\geq 9$ and $\frac{3}{4}a_2<a_2-a_1<a_2$. Also $n>0.09226\sqrt{0.5}b^{1/4}$. Now we observe
$$0.09226\sqrt{0.5}b^{1/4}<\frac{8\log(8.4034\cdot 10^{13}b^{3/2}c)\log(0.0153b^{1/2}c)}{\log(bc)\log(0.0000115c)},$$
and the right-hand side is a function decreasing in $c$ since $c>0.0000115^{-1}$. So, we can use a lower bound $c>0.25a_1^2b^3\geq b^3$ in this inequality and obtain $b<1.846\cdot 10^9$. This implies $a_2\leq 42965$. We can improve this upper bound for $a_2$ furthermore by observing for each fixed $a_1\in[8,159]$ a similar inequality with $a_2$. However, for values $a_1\in[2,7]$, this approach yields a larger bound compared to $42965$.   We do a computer search, and it returns there don't exist elements $b$ and $c$ which satisfy all of the conditions. 

\textit{ Case $a_1\geq 160$}: If $a_1\geq 160$ then $a_2\geq 643$. We use $a_2>0.0251a_1^2$, so $a_1<\frac{a_2^{1/2}}{\sqrt{0.0251}}$ and $a_2-a_1>a_2\left(1-\frac{1}{\sqrt{0.0251a_2}}\right)>0.751a_2$.

We will first prove that $b<3a_2^2$ and then we will prove that $b<a_2^2$. Thus, we will assume to the contrary that $b\geq 3a_2^2$. To generalize our approach, let's assume that $b\geq ka_2^2$. As before, it is easy to see that Lemmas \ref{lemma: gornja_n_1} and \ref{lema: donja na n u terminima b i c} hold. We have $c>0.25\cdot 160^2k^3a_2^6=6400k^3a_2^6$, $c/b^2>6400ka_2^2$  and $n>0.09226\cdot(6400k)^{1/4}a_2^{1/2}$. It is easy to see from Corollary \ref{cor: primjena leme 3.4} that we can observe the inequality
$$0.09226\cdot(6400k)^{1/4}a_2^{1/2}<\frac{8\log(2.7136\cdot 10^{18}\cdot k^3\cdot a_2^{35/4})\log(4599.392\cdot k^3\cdot a_2^{23/4})}{\log(6400\cdot k^4\cdot a_2^8)\log(4313.6\cdot k^3\cdot a_2^2)},$$
which yields $a_2\leq 533$ for $k=3$. This is a contradiction to $a_2\geq 643$, so $b<3a_2^2$. 

Now we assume that $b\geq a_2^2$ holds. Previous inequality for $k=1$ yields $a_2\leq 1140$ which implies $a_1\leq 231$. 
Then we also have $b<3a_2^2\leq 3898800$. 

Again, we consider finitely many pairs $\{a_1,a_2\}$ and search for $b$ and $c$ such that $\max\{10^5+1,a_2^2\}\leq b<3a_2^2$ and $0.25a_1^2b^3<c<0.25a_2^2b^3$, and a computer search returns that none such $D(4)$-pair $\{b,c\}$ exists.
\end{proof}

It is now easy to finish the proof of statement 3) of Theorem \ref{glavni}. Since $10^5<b<a_2^2$, when $a_1\geq 2$, we have $a_2\geq 317.$

\section{Proof of Corollary \ref{cor: konacno}}

First, we will prove a numerical upper bound for $c$ in an irregular $D(4)$-quadruple $\{a,b,c,d\}$. The proof follows \cite[Proposition 4]{duje_kon} and \cite[Lemma 13]{fil_xy4}. In the next lemma $(v_m)$ and $(w_n)$ denote sequences of solutions described as in (\ref{niz vm}) and (\ref{niz wm}) but with $a$ instead of $a_2$, i.e. as in \cite{mbt}. 
\begin{lemma}\label{lema: nume_c}
    Let $\{a,b,c,d\}$ be a $D(4)$-quadruple such that $a<b<c<d_+<d$,  $c>b^{2.8}$, $b>\max\{4a,10^5\}$. Then $c<10^{2157}$ and $m<3.65\cdot 10^{21}$.
\end{lemma}
\begin{proof}
Assume that $d>d_+$.    First, we will prove $n>c^{0.01}$.
    Assume to the contrary, that $n\leq c^{0.01}$.  Notice that for $c>b^{2.8}$ we have from Lemma \ref{lemma: epsilon} that $m<1.36n+1$. Also, since $b>10^5$ we have $c>10^{13}$. From Lemma \ref{lema: n vece 7} we know $n\geq 7$. 

    The only cases we have to consider are (\ref{rj_1}) and (\ref{rj_2}), where $m$ and $n$ are of the same parity. Then  from \cite[Lemma 5]{fil_xy4} we have $n\leq m$.

    \textbf{1)} Consider the case (\ref{rj_1}), where $v_{2m'}=w_{2n'}$  and $|z_0|=2$, $z_0=z_1$. \\
    We use $n'<c^{0.01}$.
As in \cite[Lemma 13]{fil_xy4}, the congruence 
$$\pm am'^2+sm'\equiv \pm bn'^2+tn'\ (\bmod \ c)$$
is an equality since $\max\{am'^2,sm',bn'^2,tn'\}<0.01c$. Hence, $\pm am'^2+sm'= \pm bn'^2+tn'.$
As in \cite[Lemma 11]{duje_kon}, by using $b>4a$ and $c>10^{13}$ it is easy to see that 
$$\frac{m}{n}=\frac{2m'}{2n'}>\frac{t}{s}\cdot \frac{1-\frac{bn'^2}{tn'}}{1+\frac{am'^2}{sm'}}>\sqrt{\frac{b}{a}}\cdot\frac{0.999}{\sqrt{1.001}\cdot1.001}>1.99,$$
which is a contradiction to $m<1.36n+1$.
   \par \textbf{2)} Consider the case (\ref{rj_2}), where $v_{2m'+1}=w_{2n'+1}$, $|z_0|=t$, $|z_1|=s$ and $z_0z_1>0$.\\
   As in \cite[Lemma 13]{fil_xy4}, we see that congruence 
   $$\pm \frac{1}{2}(astm'(m'+1)-bstn'(n'+1))\equiv 2r(n'-m')\ (\bmod \ c)$$
   holds. If we square it and use $s^2\equiv t^2\equiv 4\ (\bmod \ c)$ we get
   $$(2am'(m'+1)-2bn'(n'+1))^2\equiv 4r^2(n'-m')^2\ (\bmod \ c).$$
   It is easy to see that both sides are less than $c$, so we have equality, i.e.
   $$am'(m'+1)-bn'(n'+1)= \pm r(n'-m').$$
   As in \cite[Lemma 11]{duje_kon}, we use this to show
   \begin{equation}\label{jdb: 71}
   \frac{2n'+1}{2m'+1}-\sqrt{\frac{a}{b}}=\frac{\pm 4r(n'-m')+b-a}{\sqrt{b}(2m'+1)(\sqrt{b}(2n'+1)+\sqrt{a}(2m'+1))}.
   \end{equation}
   Denote the expression on the right-hand side of (\ref{jdb: 71}) by $A$. By substituting $m=2m'+1$ and $n=2n'+1$ into the inequality $m<1.36n+1$, we obtain $m'<1.36n'+0.68$. Now
   \begin{align*}
   A& <\frac{4r(m'-n')}{b(2m'+1)(2n'+1)}+\frac{b-a}{b(2m'+1)(2n'+1)} \\
   &<\frac{0.36n'+0.68}{2n'+1}\cdot 2.01\cdot \frac{1}{2m'+1}+\frac{1}{(2m'+1)(2n'+1)},
   \end{align*}
   where we have used $4r=4\sqrt{ab+4}<\sqrt{\frac{b^2}{4}+4}<2.01b$. 
      Inequality $m'<1.36n'+0.68$ also implies,
   $$\frac{2n'+1}{2m'+1}>\frac{2n'+1}{2.72n'+2.36}.$$
   Denote the right-hand side of this inequality by $B$. Since $n\geq 7$ implies $n'\geq 3$ then $B>0.65$ and $A<0.26\cdot 2.01\cdot \frac{1}{7}+\frac{1}{49}<0.1$ which yields $$0.65<\frac{2n'+1}{2m'+1}<0.5+0.1,$$
   a contradiction.

   Hence, we have $n>c^{0.01}$ when $d>d_+$. From \cite[Proposition 5.4]{mbt} we have
\begin{equation}\label{njb_log_m}
\frac{m}{\log(38.92(m+1))}<2.7717\cdot10^{12}\log^2c.
\end{equation}
Since $c^{0.01}<n\leq m$, i.e. $c<m^{100}$ then
$$\frac{m}{\log(38.92(m+1))\log^2(m)}<2.7717\cdot10^{16}.$$
This implies $m<3.65\cdot 10^{21}$ and $c<m^{100}<10^{2157}$.
\end{proof}

\begin{proof}[Proof of Corollary \ref{cor: konacno}]
Let $\{a_1,b,c,d\}$ and $\{a_2,b,c,d\}$ be $D(4)$-quad\-ru\-ples. From Corollary \ref{cor: nereg} we know that both $D(4)$-quadruples are irregular, so Lemma \ref{lema: nume_c} implies $c<10^{2157}$ and $m<3.65\cdot 10^{21}.$
Since $b>10^5$ and $b>a_2$, we have $c>0.25a_1^2b^3>b^{2.8}>a_2^{2.8}$ , so $s_2=\sqrt{a_2c+4}<c^{7/10}$. From the proof of \cite[Lemma 5]{fil_xy4} we see $v_m<cx_0s_2^{m-1}$. Now  \cite[Lemma 2.5]{mbt} implies 
 \begin{align*}
 v_m&<cx_0s_2^{m-1}<1.00317c\sqrt[4]{a_2c}s_2^{m-1}<1.00317c\sqrt[4]{c^{1/2.8}\cdot c}\cdot c^{\frac{7}{10}m-\frac{7}{10}}\\
 &<1.00317c^{0.64}\cdot c^{\frac{7}{10}m}
 <10^{10^{25}}.
 \end{align*}
 Since $v_m^2=cd+4>d$ we have $d<10^{10^{26}}$, which proves our statement -- there are only finitely many $D(4)$-triples $\{b,c,d\}$, $b<c<d$, which can be extended to two different $D(4)$-quadruples with smaller elements $a_1<a_2<b$.
\end{proof}

\begin{rem}
For further research, it would be interesting to  establish some better bounds for elements,  similar to those presented in \cite{cdf2}. The first step would be to prove a lower bound for $b$ in the form of $b > ka_2$, where $k > 1$. The key concept, observed in \cite{cdf2}, is to exclude solutions $c=c_{\nu}^{\pm }$ when $b\leq ka_2$. The sequence $\{c_{\nu}^{\pm }\}_{\nu}$ refers to a specific sequence of solutions to a Pellian equation that emerges in the problem of extending a $D(4)$-pair $\{a,b\}$ to a $D(4)$-triple $\{a,b,c\}$. An explicit formula for this sequence can be found, for example, in \cite{mbt}. However, in the case of $D(4)$-$m$-tuples, a similar approach cannot be employed since there exist multiple $c$ values of this form which fall within the bounds established in this paper. Consequently, alternative methods need to be devised to either eliminate the possibility of such $c$ values or demonstrate that $b > ka_2$ for some $k > 1$.
\end{rem}

\end{document}